\documentclass[12pt]{amsart}
\usepackage[latin1]{inputenc}
\usepackage{color}
\usepackage{enumerate}
\usepackage{amssymb}
\usepackage[all]{xy}
\usepackage{hyperref}

\newtheorem{theorem}{Theorem}[section]
\newtheorem{lemma}[theorem]{Lemma}

\newtheorem{proposition}[theorem]{Proposition}

\theoremstyle{definition}

\theoremstyle{remark}

\newtheorem{question}[theorem]{Question}

\newcommand{\R}{\mathbb{R}}
\newcommand{\N}{\mathbb{N}}

\def\cB{{\mathcal B}}

\usepackage{color}
\usepackage{enumerate}
\begin{document}
\setcounter{tocdepth}{1}


\title{A renorming characterization of Banach spaces containing $\ell_1(\kappa)$}

\author[Avil\'es]{Antonio Avil\'es}
\address[Avil\'es]{Universidad de Murcia, Departamento de Matem\'{a}ticas, Campus de Espinardo 30100 Murcia, Spain
	\newline
	\href{https://orcid.org/0000-0003-0291-3113}{ORCID: \texttt{0000-0003-0291-3113} } }
\email{\texttt{avileslo@um.es}}

\author[Mart\'inez-Cervantes]{Gonzalo Mart\'inez-Cervantes}
\address[Mart\'inez-Cervantes]{Universidad de Murcia, Departamento de Matem\'{a}ticas, Campus de Espinardo 30100 Murcia, Spain
	\newline
	\href{http://orcid.org/0000-0002-5927-5215}{ORCID: \texttt{0000-0002-5927-5215} } }	

\email{gonzalo.martinez2@um.es}

%

\author[Rueda Zoca]{Abraham Rueda Zoca}
\address[Rueda Zoca]{Universidad de Murcia, Departamento de Matem\'{a}ticas, Campus de Espinardo 30100 Murcia, Spain
	\newline
	\href{https://orcid.org/0000-0003-0718-1353}{ORCID: \texttt{0000-0003-0718-1353} }}
\email{\texttt{abraham.rueda@um.es}}
\urladdr{\url{https://arzenglish.wordpress.com}}

\thanks{Research partially supported by Fundaci\'{o}n S\'{e}neca [20797/PI/18] and by project MTM2017-86182-P (Government of Spain, AEI/FEDER, EU). The research of G. Mart\'inez-Cervantes
	was co-financed by the European Social Fund and the Youth European Initiative under Fundaci\'on S\'eneca
	[21319/PDGI/19]. The research of A. Rueda Zoca was also supported by Juan de la Cierva-Formaci\'on fellowship FJC2019-039973, by MICINN (Spain) Grant PGC2018-093794-B-I00 (MCIU, AEI, FEDER, UE), by Junta de Andaluc\'ia Grant A-FQM-484-UGR18 and by Junta de Andaluc\'ia Grant FQM-0185.}



\begin{abstract} 
A result of G. Godefroy asserts that a Banach space $X$ contains an isomorphic copy of $\ell_1$ if and only if there is an equivalent norm $|||\cdot|||$ such that, for every finite-dimensional subspace $Y$ of $X$ and every $\varepsilon>0$, there exists $x\in S_X$ so that $|||y+r x|||\geq (1-\varepsilon)(|||y|||+\vert r\vert)$ for every $y\in Y$ and every $r\in\mathbb R$. In this paper we generalize this result to larger cardinals, showing that if $\kappa$ is an uncountable cardinal then a Banach space $X$ contains a copy of $\ell_1(\kappa)$ if and only if there is an equivalent norm $|||\cdot|||$ on $X$ such that for every subspace $Y$ of $X$ with $dens(Y)<\kappa$ there exists a norm-one vector $x$ so that $||| y+r x|||=|||y|||+\vert r\vert$ whenever $y\in Y$ and $r\in\R$. This result answers a question posed by S. Ciaci, J. Langemets and A. Lissitsin, where the authors wonder whether the previous statement holds for infinite succesor cardinals. We also show that, in the countable case, the result of Godefroy cannot be improved to take $\varepsilon=0$. 
\end{abstract}

\maketitle

\section{Introduction}

The study of characterising the containment of $\ell_1$ sequences in Banach spaces has been a long standing problem in Banach Space Theory. Probably, one of the most famous characterisations in this line is Rosenthal $\ell_1$ theorem \cite{rosenthal}. Another well known characterisation, due to B. Maurey \cite{maurey}, is known in terms of a more geometric condition: a Banach space $X$ contains an isomorphic copy of $\ell_1$ if, and only if, there exists a non-zero element $u\in X^{**}$ so that
$$\Vert x+u\Vert=\Vert x-u\Vert$$
holds for every $x\in X$. This characterisation was put further during the eighthies, by succesive works of B. Maurey and G. Godefroy, which yield the following result.

\begin{theorem}[\cite{god}]\label{theo:godefroy}
Given a Banach space $X$, the following assertions are equivalent:
\begin{enumerate}
\item $X$ contains an isomorphic copy of $\ell_1$.
\item There exists an equivalent norm $|||\cdot|||$ and an element $u\in X^{**}$ so that
$$||| x+u ||| = |||x||| + 1$$
holds for every $x\in X$.
\item There exists an equivalent norm $|||\cdot|||$ such that for every finite-dimensional subspace $Y$ of $X$ and every $\varepsilon>0$ there exists $x\in S_X$ so that
$$\Vert y+r x\Vert\geq (1-\varepsilon)(\Vert y\Vert+\vert r\vert)$$
holds for every $y\in Y$ and every $r\in\mathbb R$.
\end{enumerate}
\end{theorem}

Note that the equivalence between (1) and (3) above was also obtained, with a different proof, in \cite[Theorem 4.3]{ksw}. Norms satisfying condition (3) are known as \emph{octahedral}.

In the recent preprint \cite{cll}, a characterisation in the same spirit was obtained for copies of $\ell_1(\kappa)$. In order to present such characterisation we need a bit of language. Let $X$ be a Banach space, an infinite cardinal $\kappa$ and $\alpha\in [-1,0)$. According to \cite[Proposition 2.2 and Definition 2.3]{cll}, $X$ is said to fail the \textit{$\alpha$-BCP$_{\kappa}$} if, given a subspace $Y$ of $X$ with $dens(Y)\leq\kappa$, there exists $x\in S_X$ so that
	$$\Vert y+r x\Vert\geq \Vert y\Vert+\vert \alpha\vert \vert r\vert$$ 
holds for every $y\in Y$ and every $r\in \R$.
The original formulation was given in terms of coverings of unit balls, but we prefer to keep the previous point of view for our interests.

The main result in \cite{cll} says that, given a Banach space $X$ and an infinite cardinal $\kappa$, $X$ contains an isomorphic copy of $\ell_1(\kappa^+)$ (where $\kappa^+$ stands for the succesor cardinal of $\kappa$) if, and only if, there exists an equivalent renorming on $X$ failing the $\alpha$-BCP$_{\kappa}$ for every $\alpha\in (-1,0)$. Moreover, the authors left as open question \cite[Question 6.1]{cll} whether the value $\alpha=-1$ can or not be reached. In other words, they asked the following.

\begin{question}\label{question:cll}
Let $\kappa$ be an infinite cardinal and let $X$ be a Banach space containing an isomorphic copy of $\ell_1(\kappa^+)$. Is there an equivalent renorming $|||\cdot|||$ on $X$ satisfying the following: for every subspace $Y$ of $X$ with $dens(Y)\leq \kappa$ there exists $x\in S_X$ so that
$$|||y+r x|||=|||y|||+\vert r\vert$$
holds for every $y\in Y$ and every $r\in\mathbb R$?
\end{question}

The main aim of this paper is to give a positive answer to the previous question. Actually, we prove a stronger version in the following sense.

\begin{theorem}\label{theo:maintheoremintro}
	\label{Theorem1}
	Given a Banach space $X$ and an uncountable cardinal $\kappa$, $X$ contains an isomorphic copy of $\ell_1(\kappa)$ if and only if there exists an equivalent renorming $|||\cdot|||$ on $X$ satisfying that, for every subspace $Y$ of $X$ with $dens(Y)<\kappa$, there exists $x\in S_X$ so that
	$$|||y+r x|||=|||y|||+\vert r\vert$$
holds for every $y\in Y$ and every $r\in\mathbb R$.
\end{theorem}

This solves, in particular, the problem posed by S. Ciaci, J. Langemets and A. Lissitsin \cite[Question 6.1]{cll}, in the particular case when we take $\kappa$ an infinite succesor cardinal. As we have pointed out they proved a weaker form of Theorem \ref{theo:maintheoremintro}, when $\kappa^+$ is a successor cardinal, for the $\alpha$-BCP$_{\kappa}$-property with $-1<\alpha<0$. Their proof relies on the renorming technique employed by V. Kadets, V. Shepelska and D. Werner in \cite[Theorem 4.3]{ksw}. These latter authors also mention, just after their Theorem 4.3, an alternative approach suggested to them by W. B. Johnson, that takes advantage of the fact that $L_\infty$ is 1-injective. Our proof follows that line of thought, but using different 1-injective spaces, coming from the completion of free Boolean algebras instead of measure algebras. However, we show in Proposition \ref{prop:comparacion} that, surprisingly, the renorming of \cite[Theorem 3.1]{cll} happens to be exactly the same as ours.

It could be wondered whether the cardinal $\kappa=\omega$ can be taken in Theorem~\ref{theo:maintheoremintro}, turning the condition $dens(Y)<\kappa$ into $Y$ being finite-dimensional. That would mean that we could improve Godefroy's result by making $\varepsilon=0$ in assertion (3) of Theorem \ref{theo:godefroy}. We observe in Proposition \ref{prop:differentia} that this is impossible since in every separable Banach space $X$ there exists $x\in S_X$ so that $\min\{\Vert x+y\Vert,\Vert x-y\Vert\}<2$ holds for every $y\in S_X$.

Throughout the text, we will deal with real Banach spaces. Given a Banach space $X$, we will denote by $B_X$ (respectively $S_X$) the closed unit ball (respectively the unit sphere). The rest of the necessary notation will be introduced when needed.

%
%
%
%
%
%
%
%
%
%


\section{Main results}\label{section:general}

Recall that a Banach space $X$ is \textit{1-injective} whenever every operator $T$ from a subspace of a Banach space $Y$ to $X$ can be extended to an operator $\hat{T}$ from $Y$ to $X$ with $\|\hat{T}\|=\|T\|$. A Banach space is $1$-injective if and only if it is isometric to a Banach space of continuous functions $C(K)$  with $K$ being an extremally disconnected compact Hausdorff space \cite{kelley}. These compact spaces are, in turn, characterized as those zero-dimensional compact spaces for which its clopen algebra is complete. Remember that, to any Boolean algebra $\mathcal B$ we can associate, by Stone duality, a zero-dimensional compact space $K$ whose algebra of clopen sets is isomorphic to $\mathcal B$. 
Given a subset $\mathcal{G} \subset \cB$ of a Boolean algebra $\cB$, we write $\langle \mathcal{G} \rangle$ to denote the Boolean subalgebra generated by $\mathcal{G}$. For any cardinal $\kappa$, we denote by $Fr(\kappa)=\langle \{G_\alpha: \alpha<\kappa\} \rangle$ the free Boolean algebra generated by $\kappa$ elements $G_\alpha$ (see \cite[Chapter 4]{monk}) and by $\overline{Fr(\kappa)}$ the canonical completion of $Fr(\kappa)$ \cite[Chapter 2, section 4]{monk}. The following elementary fact is related to the well-known property that the reaping number of $Fr(\kappa)$ equals $\kappa$ (see, e.g., \cite[Example 12]{monk2}). 
\begin{lemma}
\label{lemmaclopens}	
Let $\lambda<\kappa$ and $\{C_\alpha: \alpha < \lambda \}$ be a family of nonzero elements in $\overline{Fr(\kappa)}$. Then there is $\beta < \kappa$ such that $G_\beta \cap C_\alpha \neq \emptyset$ and $G_\beta^c \cap C_\alpha \neq \emptyset$ for every $\alpha<\lambda$.
\end{lemma}
\begin{proof}
By definition of the completion of a Boolean algebra \cite[Chapter 2, Definition 4.28]{monk}, $ Fr(\kappa)$ is dense in $\overline{Fr(\kappa)}$, so for every $\alpha < \lambda$ there is $C_\alpha' \in  Fr(\kappa)$ such that $\emptyset \neq C_\alpha' \subseteq C_\alpha$. Now, for every $\alpha < \lambda$ we can pick a finite set $F_\alpha$ such that $C_\alpha' \in \langle \{G_\beta: \beta\in F_\alpha\} \rangle $. Let $\beta < \kappa$ be any ordinal such that $\beta \notin \bigcup_{\alpha < \lambda}F_\alpha$. 
Now it is immediate that $G_\beta \cap C_\alpha \supseteq G_\beta \cap C_\alpha' \neq \emptyset$ and $G_\beta^c \cap C_\alpha \supseteq G_\beta^c \cap C_\alpha' \neq \emptyset$ for every $\alpha<\lambda$ as desired.
\end{proof}


This has the following consequence at the level of the space of continuous functions.

\begin{lemma}\label{lem:C(K)completo}
	Let $\kappa$ be an uncountable cardinal and $K$ be the compact topological space associated by Stone duality to $\overline{Fr(\kappa)}$. Then $C(K)$ contains a family of functions $\{f_\alpha: \alpha< \kappa\}$ isometrically equivalent to the canonical vector basis of $\ell_1(\kappa)$ which satisfies that for every subspace $Y$ with $dens(Y)< \kappa$ there exists $\beta$ so that
	$$\Vert g+ r f_\beta \Vert=\Vert g\Vert + |r|$$
	holds for every $g\in Y$ and every scalar $r\in\mathbb{R}$.
\end{lemma}

\begin{proof}
	We view the elements of $\overline{Fr(\kappa)}$ as clopen subsets of $K$. Let $f_\alpha=\chi_{G_\alpha}-\chi_{G_\alpha^c}$ for every $\alpha<\kappa$, where by $\chi_C$ we denote the characteristic function of the clopen $C$.
	Since the family $\{G_\alpha: \alpha< \kappa\}$ is independent, we have that $\{f_\alpha: \alpha< \kappa\}$ is isometrically equivalent to the canonical vector basis of $\ell_1(\kappa)$. 
	
	Pick a subspace $Y$ of $C(K)$ with $ dens(Y)=\lambda < \kappa$ and take $\{g_\alpha: \alpha<\lambda\}$ a dense subset of $Y$. Since $\kappa$ is uncountable, we can suppose without loss of generality that $\lambda\geq \omega$. For every $\alpha<\kappa$ we can take a sequence of clopens $C_\alpha^n$ such that $\|g_\alpha\|-\frac{1}{n}<|g_\alpha(t)| \leq \|g_\alpha\|$ for every $t\in C_\alpha^n$ and every $n\in \N$. Then, the family $\{C_\alpha^n: \alpha<\lambda,~n\in \N\}$ has cardinality at most $\lambda$, so by Lemma \ref{lemmaclopens} there is $\beta<\kappa$ such that $G_\beta \cap C_\alpha^n \neq \emptyset$ and $G_\beta^c \cap C_\alpha^n \neq \emptyset$ for every $\alpha<\lambda$ and every $n\in \N$.
	A routine computation shows that $f_\beta$ satisfies that $\Vert g_\alpha+r f_\beta \Vert=\Vert g_\alpha\Vert+|r|$ for every $\alpha<\lambda$.
	Now the conclusion follows from the density of $\{g_\alpha\}$ in $Y$. 
\end{proof}

Since $\overline{Fr(\kappa)}$ is complete, the space $C(K)$ is $1$-injective, and this allows to transfer the norm in Lemma \ref{lem:C(K)completo} to any Banach space containing $\ell_1(\kappa)$.

\begin{theorem}\label{theo:almostmain}
Given any uncountable cardinal $\kappa$ and any Banach space $X$ containing an isomorphic copy of $\ell_1(\kappa)$, there exists an equivalent norm $|||\cdot|||$ on $X$ such that for every subspace $Y$ of $X$ with $dens(Y)< \kappa$ there exists $x\in S_X$ so that
$$|||y+r x|||=|||y||| + |r|$$
holds for every $y\in Y$ and every $r\in \R$. Moreover, $x$ can be taken to be one of the vectors of the canonical basis of the copy of $\ell_1(\kappa)$ in $X$.
\end{theorem}

\begin{proof}
	Assume that $X$ contains an isomorphic copy of $\ell_1(\kappa)$. Renorming $X$ if necessary, we can assume that $X$ contains an isometric copy of $\ell_1(\kappa)$. We denote by $\{e_\alpha: \alpha < \kappa\}$ the canonical basis of $\ell_1(\kappa)$ inside $X$. Let $K$ be the compact topological space associated to the Boolean algebra $\overline{Fr(\kappa)}$. By Lemma \ref{lem:C(K)completo}, there exists a transfinite sequence $\{f_\alpha\}_{\alpha<\kappa}$ in $C(K)$, isometric to the canonical basis of $\ell_1(\kappa)$, such that for every subspace $Y$ of $C(K)$ with $dens(Y)<\kappa$, there exists $\beta<\kappa$ with
	$$\Vert g+ r f_\beta\Vert=\Vert g\Vert + |r|$$
	 for every $g\in Y$ and every $r \in \R$.
	We have an isometric embedding $$T\colon Z:=\overline{span}(\{e_\alpha: \alpha<\kappa\})\longrightarrow C(K)$$ such that $T(e_\alpha)=f_\alpha$ for all $\alpha<\kappa$.
	Since $C(K)$ is 1-injective, there exists a norm one extension  $T:X\longrightarrow C(K)$ that we still denote by $T$. We define a norm on $X$ by the formula
	$$|||x|||:=\Vert T(x)\Vert+\Vert [x]\Vert_{X/Z}.$$
	Let us check that this is equivalent to the original norm of $X$. On the one hand,
	$$|||x|||\leq \Vert T\Vert \Vert x\Vert+\inf\{\Vert y\Vert: y-x\in Z\}\leq \Vert x\Vert+\Vert x\Vert=2\Vert x\Vert.$$
	On the other hand, we claim that $|||x|||\geq \frac{1}{3}\Vert x\Vert$ for all $x\in X$. Otherwise, we could find $x$ such that $\Vert x\Vert=1$ but $|||x|||<1/3$. In particular we would have $\Vert [x]\Vert_{X/Z}< \frac{1}{3}$, so there exists $z\in Z$ so that $\Vert x-z\Vert<\frac{1}{3}$. Then $|||z|||=\Vert T(z)\Vert=\Vert z\Vert\geq \frac{2}{3}$ and therefore
	$$|||x|||\geq \Vert T(x)\Vert\geq \Vert T(z)\Vert-\Vert T(x-z)\Vert\geq \frac{2}{3}-\Vert x-z\Vert\geq \frac{1}{3},$$
	a contradiction. So $\Vert\cdot\Vert$ and $|||\cdot|||$ are equivalent norms, and we pass to the proof of the main statement.
	
	Take a subspace $Y$ of $X$ with $dens(Y)<\kappa$. Since $\overline{T(Y)}$ is a subspace of $C(K)$ with density character $<\kappa$, by Lemma \ref{lem:C(K)completo} we can find $\beta<\kappa$ such that
	$$\Vert T(y)+r f_\beta\Vert=\Vert T(y)\Vert + |r|$$
	for every $r\in\mathbb R$ and every $y\in Y$. Now, given $y\in Y$ and $r\in\mathbb R$, we get $$\Vert T(y)+ r T(e_\beta)\Vert=\Vert T(y)+r f_\beta\Vert=\Vert T(y)\Vert+\vert r \vert,$$ $$\Vert [y+r e_\beta]\Vert_{X/Z}=\Vert [y]\Vert_{X/Z} \text{ since } r e_\beta\in Z.$$ Applying the definition of the norm $|||\cdot|||$, we join both formulas together to obtain
	$$|||y+r e_\beta|||=\Vert T(y)\Vert + |r| + \Vert [y]\Vert_{X/Z}=|||y||| + |r|.$$
\end{proof}

Theorem \ref{theo:almostmain} already gives one implication of our Theorem \ref{Theorem1}. The converse implication is the easy one and follows from the same argument as in \cite[Proposition 23]{glm} and \cite[Proposition 3.2]{cll}. We include it here for the sake of completeness. Suppose that $(X,|||\cdot|||)$ satisfies that, for every subspace $Y$ of $X$ with $dens(Y)<\kappa$ there exists $x\in S_X$ so that $|||y+ r x|||=|||y|||+\vert r\vert$ for every $y\in Y$ and $r\in\mathbb R$. Consider $\mathcal{P}$ the poset of all subsets $A$ in $X$ isometric to the canonical basis of $\ell_1(A)$. It is immediate that $\mathcal{P}$ is nonempty (it contains singletons in the sphere). Moreover, it contains the union of any chain in $\mathcal{P}$ (ordered by inclusion). By Zorn's lemma, there exists a maximal element $A \in \mathcal{P}$. We claim that the cardinality of $A$ is at least $\kappa$. Otherwise  we can find $x \in S_X$ such that $|||y+\lambda x|||=|||y|||+|\lambda|$ for every $y \in \overline{span}(A)$ and every $\lambda \in \R$. Then $A \cup \{x\} \in \mathcal{P}$ would contradict the maximality of $A$.

Finally, notice that, in Theorem \ref{theo:almostmain} we have excluded the case $\kappa=\omega$. Actually, the following result shows that no version of Theorem \ref{theo:almostmain} can hold for $\kappa=\omega$. 

\begin{proposition}\label{prop:differentia}
Let $X$ be a separable Banach space. Then there exists a point $x\in S_X$ so that, for every $y\in S_X$, then 
$$\min\{\Vert x+y\Vert, \Vert x-y\Vert\}<2.$$
\end{proposition}

\begin{proof}
Since $X$ is separable, it is known that the norm of $X$ is G\^ateaux differentiable at some point $x\in S_X$ \cite[Theorem 8.14]{fab} which implies, by classical Smulyan lemma \cite[Lemma 8.4]{fab}, that there exists only one $f\in S_{X^*}$ so that $f(x)=1$. We claim that the point $x$ satisfies our purposes. If we assume the contrary, there exists $y\in S_X$ so that $\Vert x\pm y\Vert=2$. By Hahn-Banach theorem, take $f_{\pm}\in S_{X^*}$ so that $f_{\pm}(x\pm y)=2$. It is clear that $f_\pm (x)=1$ which implies that $f_+=f=f_-$ because of the differentiability condition on $x$. However, we get that $f(x\pm y)=2$, from where $f(\pm y)=1$, contradicting the linearity of $f$. 
\end{proof}

We finish the paper by comparing the renorming techniques given in \cite[Proposition 3.3]{cll} and in our Theorem \ref{theo:almostmain}. The former, omitting unnecessary restrictions on successor cardinals, is described as follows: let $X$ be a Banach space containing an isomorphic copy of $\ell_1(\kappa)$. Assume, with no loss of generality, that there exists a subspace $Y$ on $X$ which is isometrically isomorphic to $\ell_1(\kappa)$. Define  $\mathcal P$ as the set of those seminorms $q:X\longrightarrow \mathbb R$ satisfying that $q(y)=\Vert y\Vert$ holds for every $y\in Y$ and $q(x)\leq \Vert x\Vert$ holds for every $x\in X$. Pick a minimal element $p\in\mathcal P$. Then, the equivalent norm given in \cite[Proposition 3.3]{cll} is given by the following equation
$$|||x|||:=p(x)+\Vert [x]\Vert_{X/Y}.$$
We will prove that the previous renorming technique is covered by the renorming given in Theorem \ref{theo:almostmain} in the following sense.

\begin{proposition}\label{prop:comparacion}
Let $X$ be a Banach space containing an isometric copy of $\ell_1(\kappa)$ (say $Y:=\overline{span}\{e_\alpha\}_{\alpha\in\kappa}$) and let $p:X\longrightarrow \mathbb R$ be a minimal seminorm in $\mathcal P$.

Consider the compact topological space $K$ and the family of functions $\{f_\alpha\}_{\alpha<\kappa}$ described in Lemma \ref{lem:C(K)completo}. Then there exists a norm-one operator $T:X\longrightarrow C(K)$ such that
\begin{itemize}
\item $T(e_\alpha)=f_\alpha$ holds for every $\alpha<\kappa$.
\item $\Vert T(x)\Vert= p(x)$ holds for every $x\in X$.
\end{itemize}
\end{proposition}

As a consequence, the renorming given in \cite[Proposition 3.3]{cll} is of the form given in the proof of Theorem \ref{theo:almostmain}.

\begin{proof}
Pick a sequence $\{q_\alpha\}_{\alpha<\kappa}$ in $S_{\ell_\infty(\kappa)}$ which is isometric to the $\ell_1(\kappa)$-basis. Define $g:Y\longrightarrow \ell_\infty(\kappa)$ by the equation
$$g(e_\alpha)=q_\alpha\ \forall \alpha<\kappa.$$
By the definition of $\ell_\infty(\kappa)$ there are, for every $\alpha<\kappa$, functionals $g_\alpha:Y\longrightarrow \mathbb R$ such that $g(y)=(g_\alpha(y))_\alpha\in\ell_\infty(\kappa)$. The operator $g$ is clearly an isometric embedding since $(e_\alpha)$ and $(q_\alpha)$ are isometric to the $\ell_1(\kappa)$ basis. Consequently, we obtain that
$$\sup\limits_{\alpha<\kappa} \vert g_\alpha(y)\vert=\Vert y\Vert=p(y)$$
holds for every $y\in Y$ so, in particular $\vert g_\alpha(y)\vert\leq p(y)$ for every $y\in Y$ and every $\alpha<\kappa$. Now, for every $\alpha < \kappa$ there exists, by the Hahn-Banach theorem \cite[Theorem 2.1]{fab}, a linear map $G_\alpha:X\longrightarrow \mathbb R$ so that its restriction to $Y$ is $G_{\alpha}|_Y=g_\alpha$ and $\vert G_\alpha(x)\vert\leq p(x)$ for every $x\in X$. This allows to define $G:X\longrightarrow \ell_\infty(\kappa)$ by 
$$G(x)=(G_\alpha(x))_\alpha.$$
From the inequality $\vert G_\alpha(x)\vert\leq p(x)$ we get $\Vert G(x)\Vert\leq p(x) \leq \|x\|$ for every $x\in X$. 

Now, the fact that $C(K)$ is 1-injective implies that there exists a norm-one operator $\phi:\ell_\infty(\kappa)\longrightarrow C(K)$ satisfying that $\phi(q_\alpha)=f_\alpha$ holds for every $\alpha < \kappa$. Consider finally $T=\phi\circ G:X\longrightarrow C(K)$. On the one hand, given $\alpha\in\kappa$, we get
$$T(e_\alpha)=\phi(G(e_\alpha))=\phi(q_\alpha)=f_\alpha.$$
On the other hand, given $x\in X$, we get
$$\Vert T(x)\Vert\leq \Vert \phi\Vert\Vert G(x)\Vert\leq p(x).$$
Now if we define $q(x):=\Vert T(x)\Vert$, by the properties exposed above we get that $q\in\mathcal P$ and that $q(x)\leq p(x)$. By the minimality of $p$, we conclude that $p(x)=q(x)=\Vert T(x)\Vert$ holds for every $x\in X$, and the proof is finished.
\end{proof}


\begin{thebibliography}{99}











\bibitem {cll} S.~Ciaci, J.~Langemets and A.~Lissitsin, \textit{A characterization of Banach spaces containing $\ell_1(\kappa)$ via ball-covering properties}, preprint, arXiv:2102.13525.

\bibitem {fab} M. Fabian, P. Habala, P. H\'ajek, V. Montesinos, J. Pelant, V. Zizler, \textit{Functional Analysis and Infinite-Dimensional Geometry}, CMS Books in Mathematics, Springer-Verlag, New York, 2001.


\bibitem {god} G. Godefroy, {\it Metric characterization of first Baire class linear forms and octahedral norms,} Studia Math. {\bf 95} (1989), 1-15.



\bibitem {glm} A.~J.~Guirao, A. Lissitsin and V. Montesinos, \textit{Some remarks on the ball-covering property}, J.  Math. Anal. Appl. \textbf{479} (2019), 608--620.





	


\bibitem {ksw} V. Kadets, V. Shepelska and D. Werner, \textit{Thickness of the unit sphere, $\ell_1$-types, and the almost Daugavet property}, Houston J. Math. \textbf{37}  (2011), 867--878.

\bibitem {kelley} J.~L.~Kelley, \textit{Banach spaces with the extension property}, Trans. Amer. Math. Soc. \textbf{72} (1952), 323--326.





\bibitem {maurey} B. Maurey, \textit{Types and $\ell_1$-subspaces}, Longhorn Notes, Texas Functional Analysis Seminar, Austin, Texas 1982/1983.


\bibitem {monk} J.~D.~Monk, \textit{Handbook of Boolean algebras. Volumes 1--3}, Amsterdam etc.: North-Holland. xix, 1367 p. (1989). 

\bibitem {monk2} J.~D.~Monk, \textit{Cardinals Generalized to Boolean Algebras}, J. Symb. Logic \textbf{66}, 4 (2001), 1928--1958. 


\bibitem {rosenthal} H.~P.~Rosenthal, \textit{A characterization of Banach spaces containing $\ell_1$}, Proc. Natl. Acad. Sci. USA \textbf{71} (1974), 2411--2413.








\end{thebibliography}
\end{document}